\documentclass[11pt]{article}
\usepackage[dvips]{graphicx}
\usepackage{textcomp}
\usepackage{amsbsy}
\usepackage{latexsym}
\usepackage[mathscr]{eucal}
\usepackage{amsfonts,amsmath,amsthm}
\usepackage{amssymb}
\usepackage[usenames]{xcolor}

\usepackage{fullpage}

\begin{document}
\bibliographystyle{plain}
\title
{ A note on best $n$-term approximation for generalized Wiener classes}
\author{Ronald DeVore, Guergana Petrova,  Przemys{\l}aw Wojtaszczyk}
%
%
%
\hbadness=10000
\vbadness=10000
\newtheorem{lemma}{Lemma}[section]
\newtheorem{prop}[lemma]{Proposition}
\newtheorem{cor}[lemma]{Corollary}
\newtheorem{theorem}[lemma]{Theorem}
\newtheorem{remark}[lemma]{Remark}
\newtheorem{example}[lemma]{Example}
\newtheorem{definition}[lemma]{Definition}
\newtheorem{proper}[lemma]{Properties}
\newtheorem{assumption}[lemma]{Assumption}
%
\newenvironment{disarray}{\everymath{\displaystyle\everymath{}}\array}{\endarray}

\def\RR{\rm \hbox{I\kern-.2em\hbox{R}}}
\def\NN{\rm \hbox{I\kern-.2em\hbox{N}}}
\def\ZZ{\rm {{\rm Z}\kern-.28em{\rm Z}}}
\def\CC{\rm \hbox{C\kern -.5em {\raise .32ex \hbox{$\scriptscriptstyle
|$}}\kern
-.22em{\raise .6ex \hbox{$\scriptscriptstyle |$}}\kern .4em}}
\def\vp{\varphi}
\def\<{\langle}
\def\>{\rangle}
\def\t{\tilde}
\def\i{\infty}
\def\e{\varepsilon}
\def\sm{\setminus}
\def\nl{\newline}
\def\o{\overline}
\def\wt{\widetilde}
\def\wh{\widehat}
\def\cT{{\cal T}}
\def\cA{{\cal A}}
\def\cI{{\cal I}}
\def\cV{{\cal V}}
\def\cB{{\cal B}}
\def\cF{{\cal F}}
\def\cY{{\cal Y}}

\def\cD{{\cal D}}
\def\cP{{\cal P}}
\def\cJ{{\cal J}}
\def\cM{{\cal M}}
\def\cO{{\cal O}}
\def\Chi{\raise .3ex
\hbox{\large $\chi$}} \def\vp{\varphi}
\def\lsima{\hbox{\kern -.6em\raisebox{-1ex}{$~\stackrel{\textstyle<}{\sim}~$}}\kern -.4em}
\def\lsim{\hbox{\kern -.2em\raisebox{-1ex}{$~\stackrel{\textstyle<}{\sim}~$}}\kern -.2em}
\def\[{\Bigl [}
\def\]{\Bigr ]}
\def\({\Bigl (}
\def\){\Bigr )}
\def\[{\Bigl [}
\def\]{\Bigr ]}
\def\({\Bigl (}
\def\){\Bigr )}
\def\L{\pounds}
\def\pr{{\rm Prob}}
\newcommand{\cs}[1]{{\color{magenta}{#1}}}
\def\ds{\displaystyle}
\def\ev#1{\vec{#1}}     
\newcommand{\lt}{\ell^{2}(\nabla)}
\def\Supp#1{{\rm supp\,}{#1}}
\def\R{\mathbb{R}}
\def\E{\mathbb{E}}
\def\nl{\newline}
\def\T{{\relax\ifmmode I\!\!\hspace{-1pt}T\else$I\!\!\hspace{-1pt}T$\fi}}
\def\N{\mathbb{N}}
\def\Z{\mathbb{Z}}
\def\N{\mathbb{N}}
\def\Zd{\Z^d}
\def\Q{\mathbb{Q}}
\def\C{\mathbb{C}}
\def\Rd{\R^d}
\def\gsim{\mathrel{\raisebox{-4pt}{$\stackrel{\textstyle>}{\sim}$}}}
\def\sime{\raisebox{0ex}{$~\stackrel{\textstyle\sim}{=}~$}}
\def\lsim{\raisebox{-1ex}{$~\stackrel{\textstyle<}{\sim}~$}}
\def\div{\mbox{ div }}
\def\M{M}  \def\NN{N}                  
\def\Le{{\ell^1}}            
\def\Lz{{\ell^2}}
\def\Let{{\tilde\ell^1}}     
\def\Lzt{{\tilde\ell^2}}
\def\Ltw{\ell^\tau^w(\nabla)}
\def\t#1{\tilde{#1}}
\def\la{\lambda}
\def\La{\Lambda}
\def\ga{\gamma}
\def\BV{{\rm BV}}
\def\Ga{\eta}
\def\al{\alpha}
\def\cZ{{\cal Z}}
\def\cA{{\cal A}}
\def\cU{{\cal U}}
\def\argmin{\mathop{\rm argmin}}
\def\argmax{\mathop{\rm argmax}}
\def\prob{\mathop{\rm prob}}

\def\cO{{\cal O}}
\def\cA{{\cal A}}
\def\cC{{\cal C}}
\def\cS{{\cal F}}
\def\bu{{\bf u}}
\def\bz{{\bf z}}
\def\bZ{{\bf Z}}
\def\bI{{\bf I}}
\def\cE{{\cal E}}
\def\cD{{\cal D}}
\def\cG{{\cal G}}
\def\cI{{\cal I}}
\def\cJ{{\cal J}}
\def\cM{{\cal M}}
\def\cN{{\cal N}}
\def\cT{{\cal T}}
\def\cU{{\cal U}}
\def\cV{{\cal V}}
\def\cW{{\cal W}}
\def\cL{{\cal L}}
\def\cB{{\cal B}}
\def\cG{{\cal G}}
\def\cK{{\cal K}}
\def\cX{{\cal X}}
\def\cS{{\cal S}}
\def\cP{{\cal P}}
\def\cQ{{\cal Q}}
\def\cR{{\cal R}}
\def\cU{{\cal U}}
\def\bL{{\bf L}}
\def\bl{{\bf l}}
\def\bK{{\bf K}}
\def\bC{{\bf C}}
\def\X{X\in\{L,R\}}
\def\ph{{\varphi}}
\def\D{{\Delta}}
\def\H{{\cal H}}
\def\bM{{\bf M}}
\def\bx{{\bf x}}
\def\bj{{\bf j}}
\def\bG{{\bf G}}
\def\bP{{\bf P}}
\def\bW{{\bf W}}
\def\bT{{\bf T}}
\def\bV{{\bf V}}
\def\bv{{\bf v}}
\def\bt{{\bf t}}
\def\bz{{\bf z}}
\def\bw{{\bf w}}
\def \span{{\rm span}}
\def \meas {{\rm meas}}
\def\rhom{{\rho^m}}
\def\diff{\hbox{\tiny $\Delta$}}
\def\EE{{\rm Exp}}
\def\lll{\langle}
\def\argmin{\mathop{\rm argmin}}
\def\codim{\mathop{\rm codim}}
\def\rank{\mathop{\rm rank}}

\def\argmax{\mathop{\rm argmax}}
\def\dJ{\nabla}
\newcommand{\ba}{{\bf a}}
\newcommand{\bb}{{\bf b}}
\newcommand{\bc}{{\bf c}}
\newcommand{\bd}{{\bf d}}
\newcommand{\bs}{{\bf s}}
\newcommand{\bff}{{\bf f}}
\newcommand{\bp}{{\bf p}}
\newcommand{\bg}{{\bf g}}
\newcommand{\by}{{\bf y}}
\newcommand{\br}{{\bf r}}
\newcommand{\be}{\begin{equation}}
\newcommand{\ee}{\end{equation}}
\newcommand{\bea}{$$ \begin{array}{lll}}
\newcommand{\eea}{\end{array} $$}
\def \Vol{\mathop{\rm  Vol}}
\def \mes{\mathop{\rm mes}}
\def \Prob{\mathop{\rm  Prob}}
\def \exp{\mathop{\rm    exp}}
\def \sign{\mathop{\rm   sign}}
\def \sp{\mathop{\rm   span}}
\def \rad{\mathop{\rm   rad}}
\def \vphi{{\varphi}}
\def \csp{\overline \mathop{\rm   span}}

\def\beginproof{\noindent{\bf Proof:}~ }
\def\endproof{\hfill\rule{1.5mm}{1.5mm}\\[2mm]}

\newenvironment{Proof}{\noindent{\bf Proof:}\quad}{\endproof}

\renewcommand{\theequation}{\thesection.\arabic{equation}}
\renewcommand{\thefigure}{\thesection.\arabic{figure}}

\makeatletter
\@addtoreset{equation}{section}
\makeatother

\newcommand\abs[1]{\left|#1\right|}
\newcommand\clos{\mathop{\rm clos}\nolimits}
\newcommand\trunc{\mathop{\rm trunc}\nolimits}
\renewcommand\d{d}
\newcommand\dd{d}
\newcommand\diag{\mathop{\rm diag}}
\newcommand\dist{\mathop{\rm dist}}
\newcommand\diam{\mathop{\rm diam}}
\newcommand\cond{\mathop{\rm cond}\nolimits}
\newcommand\eref[1]{{\rm (\ref{#1})}}
\newcommand{\iref}[1]{{\rm (\ref{#1})}}
\newcommand\Hnorm[1]{\norm{#1}_{H^s([0,1])}}
\def\int{\intop\limits}
\renewcommand\labelenumi{(\roman{enumi})}
\newcommand\lnorm[1]{\norm{#1}_{\ell^2(\Z)}}
\newcommand\Lnorm[1]{\norm{#1}_{L_2([0,1])}}
\newcommand\LR{{L_2(\R)}}
\newcommand\LRnorm[1]{\norm{#1}_\LR}
\newcommand\Matrix[2]{\hphantom{#1}_#2#1}
\newcommand\norm[1]{\left\|#1\right\|}
\newcommand\ogauss[1]{\left\lceil#1\right\rceil}
\newcommand{\QED}{\hfill
\raisebox{-2pt}{\rule{5.6pt}{8pt}\rule{4pt}{0pt}}%
  \smallskip\par}
\newcommand\Rscalar[1]{\scalar{#1}_\R}
\newcommand\scalar[1]{\left(#1\right)}
\newcommand\Scalar[1]{\scalar{#1}_{[0,1]}}
\newcommand\Span{\mathop{\rm span}}
\newcommand\supp{\mathop{\rm supp}}
\newcommand\ugauss[1]{\left\lfloor#1\right\rfloor}
\newcommand\with{\, : \,}
\newcommand\Null{{\bf 0}}
\newcommand\bA{{\bf A}}
\newcommand\bB{{\bf B}}
\newcommand\bR{{\bf R}}
\newcommand\bD{{\bf D}}
\newcommand\bE{{\bf E}}
\newcommand\bF{{\bf F}}
\newcommand\bH{{\bf H}}
\newcommand\bU{{\bf U}}
\newcommand \A {{\bb A}}
\newcommand\cH{{\cal H}}
\newcommand\sinc{{\rm sinc}}
\def\enorm#1{| \! | \! | #1 | \! | \! |}

\newcommand{\am}{a_{\min}}
\newcommand{\aM}{a_{\max}}

\newcommand{\dm}{\frac{d-1}{d}}

\let\bm\bf
\newcommand{\bbeta}{{\mbox{\boldmath$\beta$}}}
\newcommand{\bal}{{\mbox{\boldmath$\alpha$}}}
\newcommand{\bbi}{{\bm i}}

\def\nnew{\color{Red}}
\def\mnew{\color{Blue}}
\def\wnew{\color{magenta}}
\def\gnew{\color{green}}

\newcommand{\dI}{\Delta}
\newcommand\aconv{\mathop{\rm absconv}}

\maketitle
\date{}
\begin{abstract}
We determine  the  best $n$-term approximation  of  generalized   Wiener model classes   in a Hilbert  space $H $. This theory is then applied to several special cases.
\end{abstract}
\section{Introduction}
\label{sec:1}
 One of the main themes in approximation theory is  to prove theorems on how well functions can be approximated
in a Banach space norm $\|\cdot\|_X$ by methods of linear or nonlinear approximation.   
The present paper is exclusively concerned with approximation in a separable Hilbert space  $H$ equipped  with  norm $\|\cdot\|$, induced by a scalar product  $\langle\cdot,\cdot\rangle$.    Let  ${\cal D}:=\{\phi_i, \,i\in\mathbb N\}$ be an orthonormal basis for $H$.
This means that any function $f\in H$ has the unique representation
$$
f=\sum_{j=1}^\infty f_j\phi_j,\quad \hbox{where}\quad \|f\|^2=\sum_{j=1}^\infty |f_j|^2.
$$
We are concerned with $n$ term approximation of the elements $f\in H$. We denote by  $\Sigma_n:=\{S=\sum_{j\in \Lambda}c_je_j:\,\Lambda\subset \mathbb N, \,|\Lambda|=n\}$
and  let 
$$
\sigma_n(f):=\inf_{S\in\Sigma_n}\|f-S\|,\quad n\ge 1,
$$
be the {\it error of $n$-term approximation} of $f$.  Given any $f\in H$, a best $n$-term approximation $S_n$ of $f$ is given by
$$
S_n=S_n(f):=\sum_{j\in\Lambda_n} f_j\phi_j,
$$
where $\Lambda_n:=\Lambda_n(f)$ is a set of $n$ indices $j$ for which $|f_j|\ge |f_i|$ whenever $j\in\Lambda$ and $i\notin\Lambda$.  Even though the set $\Lambda_n(f)$ is not uniquely defined because of possible ties in terms of the size of the absolute values of the  coefficients $f_j$,  the error of approximation $\sigma_n(f)$ is uniquely defined.

We are interested in model classes $K\subset H$ that are given by imposing conditions on the coefficients $(f_j)$ of $f$. For such sets $K$, we define
\begin{equation}
\label{ntermK}
\sigma_n(K):=\sup_{f\in K} \sigma_n(f)
\end{equation}
and we are interested in  the asymptotic decay of $\sigma_n(K)\to 0$ as $n\to \infty$.

Since we are only considering the approximation to take place in $H$, in going further it is sufficient to consider only the case
$$
H=\ell_2:=\left\{{\bf x}=(x_1, x_2, \ldots,):\,\|x\|^2_{\ell_2}:=\sum_{j=1}^\infty |x_j|^2<\infty\right\}.
$$
A classical result in this case is the following.  Let $ 0<p<2$,  and consider the unit ball in $\ell_p$,
$$
K=U(\ell_p):=\left\{{\bf x}=(x_1, x_2, \ldots,):\,\|x\|^p_{\ell_p}:=\sum_{j=1}^\infty |x_j|^p\leq 1\right\}\subset \ell_2.
$$
It is known in this case  that
\begin{equation}
\label{nellp}
\sigma_n(U(\ell_p))\asymp n^{-1/p+1/2},\quad n\to\infty,
\end{equation}
with absolute constants in the equivalency\footnote{We use the notation $A\asymp B$ when there are absolute constants$C_1,C_2>0$such that  we have $C_1B\leq A\leq C_2B$} .  This result is attributed to Stechkin \cite{S}.  

Other results of the above type have been frequently obtained in the literature.  To describe a general setting, let 
$$
{\bf w}:=(w_j)_{j\in\mathbb N},\quad 1\le w_1\le w_2\le \dots,
$$
be a monotonically nondecreasing  sequence of positive  weights. We consider the weighted $\ell_p$
space $\ell_p({\bf w})$ defined  as the set of all real valued sequences ${\bf x}\in\ell_2$ such that
$$
\ell_p({\bf w}):=\{{\bf x}\in\ell_2:\,\|{\bf x}\|_{\ell_p({\bf w})}<\infty\},\quad 0<p\le \infty,
$$
where
\begin{eqnarray}
\label{weightedlp}
\|{\bf x}\|_{\ell_p({\bf w})}:=
\begin{cases}
 \left [\sum_{j=1}^\infty  |w_jx_j|^p\right ]^{1/p}, \quad 0<p<\infty,\\\\
\sup_j|w_jx_j|, \quad \quad\quad\quad p=\infty.
\end{cases}
\end{eqnarray} 

 We denote by $U(\ell_p({\bf w}))$ the unit ball of this class
 $$
 U(\ell_p({\bf w}))):=\{{\bf x}\in\ell_p({\bf w}):\,\|{\bf x}\|_{\ell_p({\bf w})}\leq 1\}
 $$
and derive the rate of  the error  $\sigma_n(U_p({\bf w}))$, see   Theorem \ref{L:optimal}.   It can happen that $\sigma_n(U_p({\bf w}))$ is infinite for all $n$.  Also, note that the class $\ell_p({\bf w})$ is different 
from the classical weighted sequence spaces since the weights in the $\ell_p({\bf w})$ norm are  raised to a power.
 The problem considered here has been investigated in the context of best $n$-term approximation of diagonal operators in the general case of approximation in $\ell_q$. The works of Stepanets, 
see \cite{S2,S3,S4},  determine the rate of $\sigma_n(U(\ell_p({\bf w}))))_{\ell_q}$  for general $0<q\leq \infty$  under the restriction $\lim_{k\to\infty}w_k=\infty$. Recently,  in   \cite{NN}, based on results from \cite{G}, 
 the condition on the sequence ${\bf w}$ in the case $0<p<q$ has been removed, see Theorem 2.1(i). The case   $q<p<\infty$ has also been considered, but under additional assumptions on the weight sequence ${\bf w}$, see 
 Theorem 2.1(ii). 
 In this paper, we consider only the case $q=2$ and provide a simple, different unified approach for  finding the rate of  $\sigma_n(U(\ell_p({\bf w}))))_{\ell_2}$ for all $0<p\leq \infty$ with no restriction on the weight ${\bf w}$.

 We then go on to apply our result in the case of several special weights  
  $$
w_j=j^{\alpha}(1+\log j)^{\beta},  \quad j\in \mathbb N,
 $$
   provided  $\alpha>0, \,\beta\in\mathbb R$, or $\alpha=0, \,\beta\geq 0$.
 We show  in Corollary 
\ref{C2} that when  $\alpha=0$, $\beta\geq 0$,
$$
\sigma_n( U(\ell_p({\bf w}))\asymp n^{-(1/p-1/2)}[\log n]^{-\beta}, \quad n>1, \quad 0<p<2.
$$
In the case $\alpha=\beta=0$, this  recovers Stechkin's result (\ref{nellp}). In Corollary 
\ref{C3}, we consider the case $\alpha>0$, $\beta\in\mathbb R$,  $0<p\leq \infty$ and show that 
$$
 U(\ell_p({\bf w})))\asymp n^{-(\alpha+1/p-1/2)}[\log n]^{-\beta}, \quad  n>1, \quad \alpha+1/p-1/2>0. 
$$

 We  call  $\ell_p({\bf w})$  a {\it generalized Wiener class} in analogy with the definition of Wiener spaces in Fourier analysis when $H=L_2([0,1])$ and $\phi_j$ is the Fourier basis.  
 Our results have some overlap with the study of Wiener classes in the Fourier setting, that is, when $\cal D$ is the Fourier basis ${\cal F}$.   When considering the specific case of Fourier basis,  our results, which are restricted to  approximation in Hilbert spaces, are valid for $L_2$.    Several results in the literature consider the approximation of Wiener classes in $L_q$
 when the basis is the Fourier basis ${\cal F}$.
For example,  the case  $\alpha>1/2$, $\beta=0$,  and $p=1$ has been analyzed in \cite{JUV} and   upper bounds for the error in 
$L_\infty$, have been obtained  for  the  Wiener spaces with ${\cal D}$ being  the $d$-dimensional Fourier basis  ${\cal F}^d$, see Lemma 4.3(i). Recently, these results have been improved in \cite {M}, 
where matching up to logarithm lower and upper bounds for multidimensional Wiener spaces  with ${\cal D}={\cal F}^d$ 
are given for the case $\beta=0$, $\alpha>0$, $0<p\leq 1$, see Corollary 4.3 in \cite{M}, and 
$\beta=0$, $\alpha>1-1/p$, $1<p\leq q$,  all when  the error is measured in $L_q$, $2\leq q\leq \infty$, see Theorem  4.5 in \cite{M}. In 
particular, when  the dimension $d=1$ and the error is measured in the Hilbert space norm
(i.e. $q=2$), 
the results from \cite{M} give the rate  
$\sigma_n(U_p({\bf w},{\cal F}))_{L_2}\asymp n^{-(\alpha+1/p-1/2)}$, provided
 $$
 \alpha>0, \quad 0<p\leq 1, \quad \hbox{or}\quad \alpha>1-1/p, \quad 1<p\leq 2.
 $$
 This latter result is a special case of our analysis.
 

\section {Best $n$ term approximation  for $U(\ell_p({\bf w}))$} 

In going forward, we assume $H=\ell_2$ with its canonical basis $e_j$, $j\in\mathbb N$.  Before presenting our main theorem, let us introduce the 
decreasing rearrangement 
$$
{\bf x^*}=( x_j^*)_{j\in\mathbb N}
$$
 of  the absolute values of the coordinates $x_j$ of a sequence ${\bf x}=(x_j)_{j\in\mathbb N}$ that is an element of the  sequence space 
${\bf c_0}$ (consisting of all sequences whose elements converge to $0$).  Namely,  we have that $x_1^*$ is the largest of the numbers $|x_j|$, $j\in\mathbb N$,  then  $x_2^*$ is the next largest, and so on.  It follows that 
$$
x_1^*\geq x_2^*\geq \ldots,
$$
and $\|{\bf x}\|_{\ell_p}= \|{\bf x}^*\|_{\ell_p}$ for all $0<p\le\infty$.  
For each $n\ge 1$ and ${\bf x}\in\ell_2$, we have that 
\begin{equation}
\sigma_n({\bf x})=\left [\sum_{j>n} [x_j^*]^2\right]^{1/2},
\end{equation}
where $\sigma_n({\bf x})$ is the error of $n$-term approximation of ${\bf x}$ in the $\ell_2$ norm.  
In order to prove our main result, we will need the following lemma.
\begin{lemma} 
\label{L:rearrange}
If  ${\bf x}\in \ell_p({\bf w})$, then ${\bf x}^*\in \ell_p({\bf w})$ , $\sigma_n({\bf x})=\sigma_n({\bf x}^*)$, and
\begin{equation}
\|{\bf x}^*\|_{\ell_p({\bf w})} \le  \|{\bf x}\|_{\ell_p({\bf w})} .
\end{equation}
 \end{lemma}
\begin{proof}  
It follows directly from the definitions of $\sigma_n$  and ${\bf x}^*$ that  $\sigma_n({\bf x})=\sigma_n({\bf x}^*)$.
Let the sequence  ${\bf x}=(x_1,x_2,\dots)$ be in $\ell_p({\bf w})$.  We can assume that
all $x_j$ are non-negative since changing the signs of its entries does not effect neither its rearrangement nor its $\ell_p({\bf w})$ norm.
We shall first construct sequences ${\bf y}^{(m)}=(y_1^{(m)},y_2^{(m)},\dots)$, $m=0,1,\dots$,  such that each ${\bf y}^{(m+1)}$ is gotten by swapping the positions of two of the entries in  ${\bf y}^{(m)}$ with the indices of these entries each larger than $m$.  Also, the first $m$ entries of ${\bf y}^{(m)}$ satisfy
$$
y_j^{(m)}=x_j^*, \quad j=1,\dots,m.
$$
Indeed, we start with  ${\bf y}^{0}={\bf x}$.  Assuming that ${\bf y}^{(m)}$ has been defined,  we let $j>m$ be the  smallest index larger than $m$ 
such that $y^{(m)}_j$  is the largest of the entries $y^{(m)}_k$, $k>m$.  We swap the entries with positions $m+1$ and $j$ to create the sequence ${\bf y}^{(m+1)}$ from ${\bf y}^{(m)}$.   
We  have for  $m=0,1,2,\ldots,$
$$
\|{\bf y}^{(m+1)}\|_{\ell_p({\bf w})} \le \|{\bf y}^{(m)}\|_{\ell_p({\bf w})} \leq \ldots\leq  
\|{\bf y}^{(0)}\|_{\ell_p({\bf w})} 
=\|{\bf x}\|_{\ell_p({\bf w})},  
$$
because the weights in ${\bf w}$ are non-decreasing.   
Note that for every $m$
$$
\sum_{j=1}^m[w_jx^*_j]^p=\sum_{j=1}^m[w_jy^{(m)}_j]^p\leq \|{\bf y}^{(m)}\|^p_{\ell_p({\bf w})}\leq  \|{\bf x}\|^p_{\ell_p({\bf w})}, 
$$
and therefore 
$$
\|{\bf x}^*\|_{\ell_p({\bf w})}\leq  \|{\bf x}\|_{\ell_p({\bf w})},
$$
 which  completes the proof.
 \end{proof}

Now we are ready to determine  the rate of $\sigma_n(U(\ell_p({\bf w}))$ for all $n$ and all weight sequences ${\bf w}$. We fix ${\bf w}$, $n$, and $0<p<\infty$.  
 From the sequence ${\bf w}$, we define  the numbers
    \begin{equation}
  W_m:=[w^p_1+w^p_2+\ldots+w^p_m]^{1/p},\quad m\ge 1.
  \end{equation}
  Then the following theorem holds
  
  \begin{theorem}
\label{L:optimal}
For any $0<p<\infty$ and ${\bf w}$, we have
\begin{equation}
\label{Topt}
 \max_{m\geq n}\,(m-n)[W_m]^{-2} \le \sigma_n^2(U(\ell_p({\bf w})))\le  \max_{m\ge n}\,(m-n+1)[W_m]^{-2},
\end{equation}
and for $p=\infty$ we have that 
$$
\sigma_n^2(U(\ell_\infty({\bf w})))\asymp \sum_{j=n+1}^\infty w_j^{-2}.
$$
\end{theorem}
\begin{proof} 
We fix, $n,p,{\bf w}$.  For every sequence ${\bf x}\in U(\ell_p({\bf w}))$,  we consider its decreasing rearrangement ${\bf x}^*$, 
which according to Lemma \ref{L:rearrange} is also an element of the unit ball and has the same $n$-term approximation. We next construct a new sequence $\tilde {\bf x}$ from ${\bf x}^*$  
by making its first $n$  entries equal to $x^*_n$ and not touching the rest of the sequence, that is,
$$
\tilde x_j=\begin{cases}
x^*_n, \quad j=1,\ldots,n,\\
x^*_j, \quad j>n.
\end{cases}
$$
Note that because the sequence ${\bf x}^*$ is nonincreasing and the weights are nondecreasing, we have
$$
\|{\bf \tilde x}\|_{\ell_p({\bf w})}\le \|{\bf x}^*\|_{\ell_p({\bf w})}\le 1, \quad \hbox{and}\quad \sigma_n({\bf \tilde x})=\sigma_n({\bf x}^*)=\sigma_n({\bf x}).
$$ 
Let us denote by $b:=x^*_n$ and notice that  $\sum_{j=1}^n w_j^p\tilde x_j^p= W^p_n b^p$.
We have
$$
\sigma_n({\bf x})=\sigma_n({\bf \tilde x})=\sum_{j=n+1}^\infty [\tilde x_j]^2=\sum_{j=n+1}^\infty [x^*_j]^2.
$$
We  are now interested in how to change the  tail of $\tilde {\bf x}$, that is,  how to change the $x^*_j$'s with $j>n$ to new quantities $y_j$, $j>n$ so  that we  maximize 
the $n$-term approximation $\sigma^2_n({\bf y})=\sum_{j=n+1}^\infty [y_j]^2$, 
under  the restrictions $b=x^*_n\ge y_{n+1}\ge \cdots$ and 
\begin{equation} 
\sum_{j=n+1}^\infty w_j^py_j^p \leq 1-W^p_nb^p=:S.
\end{equation}
Notice that an investment of $w_j^py_j^p$ towards $S$ gives a return $y_j^2$ at coordinate $j$.  Since the $w_j$  are non-decreasing, to maximize $\sigma_n({\bf y})$,
it is best to invest as much as we can for $j=n+1$, $j=n+2$,  and so on.  So, the sequence which will maximize
$\sigma^2_n({\bf y})$ has to have  $y_j=b$, $j=n+1,\dots$ until we have used up our capital $S$. In other words, given that our sequence ${\bf y}$ 
has first $n$ coordinates $b$,  then to maximize $\sigma_n({\bf y})$ we should  take ${\bf y}=(b,b,\dots,b,c,0,0,\ldots)\in U(\ell_p({\bf w}))$,  where $0\leq c<b$. The membership in the unit ball of $\ell_p({\bf w})$ requires that 
$$
b^p\sum_{j=1}^mw_j^p+c^pw_{m+1}^p\leq 1\quad \Rightarrow\quad b^p[W_m]^p\leq 1
$$
and 
$$
\sigma^2_n({\bf y})=(m-n)b^2+c^2< (m-n+1)b^2\leq (m-n+1)[W_m]^{-2}.
$$
 Therefore, we have
$$
 \sigma^2_n(U(\ell_p({\bf w})))\leq \sup_{m\geq n}(m-n+1)W_m^{-2}.
$$
Next, consider the special sequence ${\bf s}^{(m)}$ with entries ${\bf s}_j^{(m)}$ given by
 $$
 {\bf s}_j^{(m)}=
 \begin{cases}
 W_m^{-1}, \quad j=1, \ldots,m,\\
 0, \,\quad \quad j>m.
 \end{cases}
 $$
 Clearly ${\bf s}^{(m)}\in U(\ell_p({\bf w}))$ for all $m$ and
 $$
 \sigma^2_n({\bf s}^{(m)})=(m-n)[W_m]^{-2}, \quad m>n.
 $$
  Such sequences provides the lower bound in the case $1<p<\infty$, and thus (\ref{Topt}) is proven.
  
  In the case $p=\infty$, we have that the entries $x_j$ of a sequence ${\bf x}\in \ell_\infty({\bf w})$ satisfy
  $$
  |x_j|\leq \|{\bf x}\|_{\ell_\infty({\bf w})} w_j^{-1}\quad \Rightarrow\quad \sigma^2_n({\bf x})\leq \left [\sum_{j=n+1}^\infty w_j^{-2}\right ]\|{\bf x}\|_{\ell_\infty({\bf w})}.
  $$
  On the other hand, for  the sequence 
  $$
  {\bf w}^{-1}:=(w_1^{-1}, w_2^{-1}, \ldots)\in U(\ell_\infty({\bf w})),
  $$
   we have that  $ \sigma^2_n({\bf w}^{-1})=\sum_{j=n+1}^\infty w_j^{-2}$, and the proof is completed.
  \end{proof}

\section{Special cases of sequence spaces}
In this section, we discuss several special cases of sequences ${\bf w}$ that are used in the definition of the classical Wiener spaces, see \cite{M,NNS} and the references therein.  
\begin{cor}
\label{C2}
Consider the classes $\ell_p({\bf w})$ with $w_j:= (1+\log j)^{\beta}$,  $\beta\geq 0$, $0<p<2$.   Then we have
$$
\sigma_n(U(\ell_p({\bf w}))) \asymp m^{-(1/p-1/2)}[\log m]^{-\beta}.
$$
\end{cor}
\begin{proof} Let us start by calculating the $W_m$, 
 $$
 W^p_m=\sum_{j=1}^mw_j^{p}=\sum_{j=1}^m (1+\log j)^{\beta p}\geq \frac{m}{2}(1+\log(m/2))^{\beta p}.
 $$
  But we also have $W_m^p \leq m(1+\log m)^{\beta p}
$, so for $m$ sufficiently big  we get 
\begin{equation}
\label{BB}
 W_m \asymp m^{1/p}[\log m]^{\beta }.
\end{equation}
Using this in Theorem \ref {L:optimal}  gives that 
$$
\sigma_n(U(\ell_p({\bf w})))\asymp n^{1/2-1/p}[\log n]^{-\beta},
$$
provided $0<p<2$.   
\end{proof}

\begin{cor}
\label{C3}
Consider the classes $\ell_p({\bf w})$, $0<p\leq \infty$, with   
$$w_j:= \max_{1\le i\le j} \,i^{\alpha}\log(i+1)^\beta, \quad j\in \mathbb N.$$
 with $\alpha>0$, $\beta\in \mathbb R$. Then we have
$$
 \sigma_n(U(\ell_p({\bf w}))) 
 \asymp n^{-(\alpha+1/p-1/2)}[\log (n+1)]^{-\beta}, 
$$
provided $ \alpha+1/p-1/2>0$.
\end{cor}
\begin{proof}  

\noindent
{\bf Case $0<p<\infty$:}    Let us observe that for any $\delta>0$ and $\gamma\in \mathbb R$, the function $\varphi(x):=x^\delta\log(x+1)^\gamma$,  is an increasing function  on $[c,\infty)$, if  $c=c(\delta,\gamma)$ is sufficiently large.  Therefore, we have that 
$$
 W_m^p=\sum_{j=1}^m j^{\alpha p}\log(j+1)^{\beta p}\asymp   m^{\alpha p+1}\log(m+1)^{\beta p},\quad m\ge M,
$$
provided $M$ is sufficiently large.  This gives
\begin{equation}
\label{BB1}
 W_m\asymp m^{\alpha +1/p}\log(m+1)^{\beta },\quad m\ge M.
\end{equation}
 Theorem \ref{L:optimal}  now gives that 
$$
\sigma_n(U(\ell_p({\bf w}))) 
 \asymp n^{1/2-\alpha-1/p}[\log (n+1)]^{-\beta},
 $$
 for $n$ sufficiently large as desired.
 
\noindent
{\bf Case $p=\infty$:}  In this case we have the restriction that $2\alpha>1$. According to Theorem \ref{L:optimal}, we have that 
$$
\sigma^2_n(U(\ell_\infty({\bf w})))\asymp \sum_{j=n+1}^\infty j^{-2\alpha}[\log(j+1)]^{-2\beta}\asymp n^{-2\alpha+1}\log(n+1)^{-2\beta}.
$$
\end{proof}

\begin{remark}
\label{R1}
Note that when $\beta=0$, the ranges of $\alpha$ and $p$ in Corollary \ref{C3} are
$$
\alpha>0, \quad 0<p\leq 2\quad \hbox{or}\quad  \alpha>1/2-1/p>0, \quad 2\leq p\leq \infty.
$$
\end{remark}

{\bf Acknowledgment:} This work  was supported by the  NSF Grant  DMS 2134077  and the ONR Contract N00014-20-1-278. 
We would like to thank   V. K. Nguyen for pointing out to us several valuable references.

\end{document}